\documentclass[a4paper, 11pt]{amsart}
\input xy   
\xyoption{all}
\swapnumbers
\usepackage{amssymb}
\usepackage{amsthm}
\usepackage{enumerate}
\numberwithin{equation}{section}
\theoremstyle{plain}
  \begingroup
        \newtheorem{theorem}[equation]{Theorem}
        
        \newtheorem{proposition}[equation]{Proposition}

	\newtheorem{definition}[equation]{Definition}

  \endgroup

\theoremstyle{definition}
  \begingroup
        \newtheorem{remark}[equation]{Remark}
        \newtheorem{example}[equation]{Example}
        
        \newtheorem{sinnadastandard}[equation]{}
        
        \newtheorem{construction}[equation]{Construction}
  \endgroup

\newcommand{\mr}[1]{\buildrel {#1} \over \longrightarrow}
\newcommand{\ml}[1]{\buildrel {#1} \over \longleftarrow}

\newcommand{\cc}{\mathcal}

\newcommand{\nn}{\mathbf}

\newcommand{\pa}{\cdot}

\newcommand{\cqd}{\hfill$\Box$}

\begin{document}

% $$\colim{U \in \cc{F}_p}{X_U}$$ 
% $$\colimite{U \in \cc{F}_p}{X_U}$$

\title{Spans, simplicial families and the fundamental progroupoid}

\author{Eduardo J. Dubuc}

%\address[uba]{Departamento de Matematicas, F.C.E. y N., UBA, Buenos
%  Aires, Argentina}

\begin{abstract}
% 
%which in addition satisfy a condition valid on all locally constant object when the topos is locally connected. 
%We show there that the category $\cc{G}_\cc{U}$ of covering projections trivialized by a fix cover $\cc{U}$ is an atomic topos with points. This determines a progroupoid of localic groupoids suitable indexed by a filtered poset of covers, which generalize the known results on the fundamental progroupoid of a locally connected topos to general topoi.
%
In this paper we consider simplicial families, that is, simplicial
objects indexed by a 
simplicial set.  We develop a method to construct family hypercover refinements of a cover family based on the notion of \emph{n-spans} that we introduce here. 
%the simplicial set is the index of a family hypercover
%refinement $\cc{H} = H_\bullet \mr{} \gamma^{*} S_\bullet$  of the cover family $\cc{U} = U \mr{} \gamma^*S$ on which the object is constant. 
%In particular, we
%show that any locally constant object in a locally connected topos is
%constructed by descent from a descent datum on a family of sets.  
%
In \cite{D3} we had
introduced the notion of \emph{covering projection} in a topos. They are locally constant objects satisfying an additional condition which is valid in all locally constant objects when the topos is locally connected, and developed the theory of the fundamental groupoid of a general topos. 
Here we show that covering projections can be obtained as 
objects constructed from a descent datum of a simplicial set
on a family of sets.
We construct a groupoid $\nn{G}_\cc{H}$ such that  the 
category  $\cc{G}_\cc{H}$
of covering projections trivialized by 
$\cc{H}$ is its classifying topos.
This determines a protopos $\{\cc{G}_\cc{H}\}_\cc{H}$ and a progroupoid
$\{\nn{G}_\cc{H}\}_\cc{H}$, 
suitable indexed by a filtered poset of 
hypercovers. Then we show that this progroupoid classifies torsors.   
This construction is novel also in the case of locally connected topoi, showing that locally constant object in a locally connected topos are
constructed by descent from a descent datum on a family of sets. The salient feature that
distinguishes locally connected topoi is that the progroupoid is \emph{strict}, that is, the transition morphisms are surjective on triangles, or, equivalently, the transition inverse image functors in the underlying indcategory are full and faithful. 
\end{abstract}

\maketitle

\tableofcontents
 
{\sc introduction.}
Given a locally connected topos $\cc{E} \mr{\gamma} \cc{S}et$, the
category $\cc{P}_\cc{U}$ 
of locally constant objects trivialized by a fix cover $\cc{U}$ is the
classifying topos of a groupoid $\nn{G}_\cc{U}$, $\cc{P}_\cc{U} =
 \beta\nn{G}_\cc{U}$. This determines a strict progroupoid
 $\pi_1(\cc{E})$ suitable
 indexed by a filtered poset of covers. This progroupoid is the
 fundamental progroupoid of $\cc{E}$ in the sense that for any
 (discrete) group $K$ it classifies $K$-torsors (\cite[expose IV]{G2}, \cite{M2}, \cite{B},  and  for a 
resume of the theory \cite[Appendix]{D2}).

For non-locally connected topos the category $\cc{P}_\cc{U}$ is not
the classifying topos of a groupoid, it is not even an atomic topos. In \cite{D3} we 
introduced the notion of \emph{covering projection} (locally
constant objects with an additional property) and show how they can be used in place of locally constant objects to develop the theory for an arbitrary topos. 
%
%that the category $\cc{G}_\cc{U}$ of covering projections trivialized by a fix cover $\cc{U}$ is an atomic topos with points. 
%
However this is done using localic groupoids and the sophisticated results of \cite{JT}.
%  
%Then by the results
%of \cite{JT} it is is the
%classifying topos of  localic groupoid $\nn{G}_\cc{U}$, $\cc{G}_\cc{U} = \beta\nn{G}_\cc{U}$. This determines a progroupoid on localic groupoids $\pi_1(\cc{E})$ which generalize the result on the fundamental progroupoid of a locally connected topos to general topoi. 
%

In this paper we simplify the theory by considering  \mbox{\emph{hypercovers},} in fact, family hypercover refinements of a cover family. When the topos is locally connected simplicial objects have a canonical indexing simplicial set 
\mbox{$S_\bullet = \gamma !(H_\bullet)$, $H_\bullet \mr{} \gamma^{*} S_\bullet$.} For non locally connected topoi the indexing has to be given explicitly as part of the datum, which leads to the notion of simplicial families. We show that the category $\cc{G}_\cc{H}$ of covering projections trivialized by a 
family hypercover is a descent topos on the indexing simplicial set, and as such the classifying topos of a groupoid whose existence does not depend of the results in \cite{JT}. An essential part of this paper is the construction of suitable family hypercover refinements of a cover family. We discover that the $n$-simplexes of a simplicial family furnish a notion of $n$-spans which is intimately related to the coskeleton functor, and which we use to construct these hypercover refinements.

\vspace{1ex}

We describe now in detail the contents of the paper.  
In Section 1  we define and examine the notion of \emph{n-spans}, 1-spans are the usual spans. In Section 2  we review certain aspects of simplicial sets. In Section 3 we develop the concept of \emph{simplicial families}, that is, simplicial
objects in $\cc{E}$ indexed by a simplicial set, $H_\bullet \to \gamma^*(S_\bullet)$, and establish a correspondence between simplicial families and collections of n-spans associated to the n-simplices. In Section 4 we develop the concept of \emph{family hypercover} refinements of a cover family. The principal result in this section is the construction of hypercover refinements of a cover family determined by a sets of objects and a sets of 1-spans. 
In Section 5 we construct and study certain groupoids  associated to a simplicial family, and in Section 6 we establish results relating descent data on a simplicial family with left actions of the associated groupoid. In Section 7 we recall the notion of \emph{covering projection} introduced in \cite{D3} and establish its basic properties. Then we prove the principal results of the paper, namely (1): For any covering projection $X$ constant on a cover family $\cc{U} = U \mr{} \gamma^*S$ there is family hypercover refinement 
$\cc{H} = H_\bullet \to \gamma^*(S_\bullet)$ such that $X$ is constant on $\cc{H}$, and (2): The category of covering projections constant on $\cc{H}$ is the topos of left action of a groupoid associated to the family hypercover.  Finally, in Section 8 we apply all this to the construction of the fundamental progroupoid of a general topos.

I thank Matias de Hoyo for several fruitful
conversations on the coskeleton construction. 

\vspace{1ex} 

{\sc context.} 
Throughout this paper $\cc{S} = Sets$ denotes the topos of sets, and \emph{topos} means Grothendieck topos $\cc{E} \mr{\gamma} \cc{S}$. We
argue in a way that should be valid if $\cc{S}$ is
an arbitrary topos, but since we do not use \emph{change of base}, we let the interested reader verify this. 

\section{Spans}

A \emph{1-span} is a diagram of the form 
$$
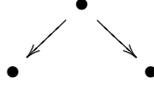
\xymatrix@C=3ex@R=3ex
         {
         {} & \bullet \ar[dl] \ar[dr]
         \\
         \bullet & {} & \bullet
         }
$$

\vspace{1ex}

A \emph{2-span} is a commutative diagram of the form
$$
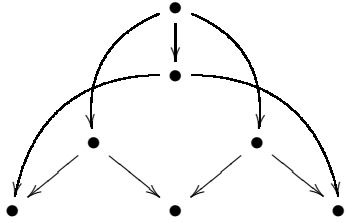
\xymatrix@C=4ex@R=3ex
         {
          {} & {} & \bullet
          \ar[d]\ar@/^1.1pc/[rdd] \ar@/_1.1pc/[ldd] 
         \\
          {} & {} & \bullet \ar@/^1.5pc/[rrdd] \ar@/_1.5pc/[lldd]
         \\
            {} & \bullet \ar[ld] \ar[rd] 
          & {} & \bullet \ar[ld] \ar[rd]
         \\
          \bullet & {} & \bullet  & {} & \bullet \, 
         }
$$
The \emph{dual} span is the span resulting from the symmetry respect
the vertical axle.  
%A span is \emph{symmetric} if it is equal to its dual span. 
For example, the dual span of the span
$
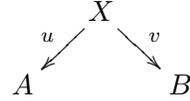
\xymatrix@C=3ex@R=3ex
         {
         {} & X \ar[dl]_u \ar[dr]^{v}
         \\
         A & {} & B
         }
$
\hspace{1ex} is \hspace{1ex}
$
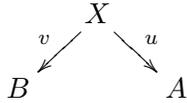
\xymatrix@C=3ex@R=3ex
         {
         {} & X \ar[dl]_v \ar[dr]^{u}
         \\
         B & {} & A
         }
$
In the same way we define the dual of a 2-span to be 2-span resulting
from the symmetry respect the vertical axle. We stress the fact that
spans are \emph{ordered from left to right} structures. 

\vspace{1ex}

A \emph{n-span} is the commutative diagram resulting from the
following procedure: At the top vertex stands a generic n-simplex (see section
\ref{ssets}), then draw an arrow to each of its $n+1$ faces
(considered ordered by the index). Then, repeat this procedure. At the
bottom level stands the $n+1$ vertices of the generic n-simplex.

\section{Simplicial sets} \label{ssets}

We fix some notation about simplicial sets, denoted by $S_\bullet$. A
simplicial set has faces $\partial_i$ and degeneracies $\sigma_i$ as follows:
$$
\xymatrix@C12ex
         {
           S_n \ar@<1ex>[r]^{\partial_i} 
         & S_{n-1} \ar@<1ex>[l]_{\sigma_i} 
         },
\;\;\; n = 0, \, 1,\, \dots \, \infty, \;\;\; \{\partial_i\}_{0 \leq i \leq n}, 
        \;\;\; \{\sigma_i\}_{0 \leq i \leq n-1}
$$
subject to the usual equations.

We write $I = S_0$, the set of vertices or $0$-simplexes. Given a 1-simplex $\ell \in S_1$ and  $i, j \in I$,  we write $i
\mr{\ell} j$ to mean $i = \partial_1(\ell)$, $j = \partial_0(\ell)$. 

Given a 2-simplex $w \in S_2$, we write
$$
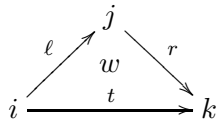
\xymatrix@R=1ex
         {
          {} & j \ar[rdd]^r
         \\
          {} & w
         \\
          i \ar[ruu]^\ell \ar[rr]^t & {} & k
         }
$$
to mean that $\partial_2(w) = \ell$, $\partial_1(w) = t$, $\partial_0(w)
= r$. We define $i = \varrho_2(w)$, $j  = \varrho_1(w)$, $k = \varrho_0(w)$ to
be the three pairs of equal composites of faces.
The simplicial equations show that this 
fits correctly. We say that the pair $i
\mr{\ell} j \mr{r} k\:$ \emph{compose}.
Notice that we have:
$$
\xymatrix@R=1ex
         {
          {} & j \ar[rdd]^{\sigma_0(j)}
         \\
          {} & \sigma_1(\ell)
         \\
          i \ar[ruu]^\ell \ar[rr]^\ell & {} & j
         }
\hspace{10ex} 
\xymatrix@R=1ex
         {
          {} & i \ar[rdd]^\ell
         \\
          {} & \sigma_0(\ell)
         \\
          i \ar[ruu]^{\sigma_0(i)} \ar[rr]^\ell & {} & j
         } 
$$
Recall that a category can be seen as a simplicial set such that given
any pair  $i \mr{\ell} j \mr{r} k$ there is a unique $w \in S_2$ 
such that $\partial_2(w) = \ell$, $\partial_0(w) = r$.
%
% and that the triangle with boundary 
%$\partial(w) =  (r,
% t,\, \ell)\:$ \emph{commutes}.
% 
%Recall:

\vspace{1ex}

We recall now the construction of a category and a groupoid associated to a
simplicial set, which involve only the first three terms. 
$$
 \xymatrix@R=10ex@C=12ex
   {
     {S}_2 \ar@<5.4ex>[r]^{\partial_0}  
           \ar[r]^{\partial_1}
           \ar@<-5.4ex>[r]^{\partial_2}  
   &  {S}_1  \ar@<2.8ex>[r]^{\partial_0} 
             \ar@<-2.8ex>[l]_{\sigma_0}
             \ar@<2.8ex>[l]_{\sigma_1} 
             \ar@<-2.8ex>[r]^{\partial_1}
   & {S}_0 \ar[l]_{\sigma_0}
  } 
$$

\begin{proposition} [Fundamental category and groupoid of a simplicial
    set, \cite{GZ}]\label{fundamentalcat}  ${}$

Objects: The set of objects is the set $I = S_0$.
%
%Premorphisms: A premorphism $i \mr{\phi} j$ is a sequence
%$(i_0 \: \ell_0 \: i_1 \: \ell_1 \: \ldots \: i_{n-1} \: \ell_{n-1} \:
%i_n)$, where  $i_0 =
%i$, $i_n = j$, $i_k \in S_0$, $k = 0,\, 1,\, \,  \ldots \, n$, $\ell_k
%\in S_1$, $k = 0,\, 1,\, \,  \ldots \, n-1$,  
%$\partial_0(l_k) = i_k$, $\partial_1(l_k) = i_{k+1}$.
%

Premorphisms: Basic premorphisms $i \mr{\ell} j$, $i,\, j \in
S_0$,  are 1-simplexes
$\ell \in S_1$, $\partial_0(\ell) = j, \; \partial_1(\ell) = i$. A general
premorphism $i \mr{\phi} j$ is a sequence
$\phi = (\ell_n  \: \ldots \:  \ell_{2} \: \ell_1)$, $\ell_k \in S_1$,
$\partial_1(\ell_1) = i$, 
$\partial_1(l_{k+1}) = \partial_0(\ell_{k})$, $n \geq 1$, \mbox{$1 \leq
  \, k \, \leq n-1$}, 
$\partial_0(\ell_{n}) = j$. When $n = 1$ we write $(\ell_1) = \ell$. 
Premorphisms compose by concatenation.

\vspace{1ex}

Morphisms: The set of morphisms is the quotient of the set of
premorphisms by the equivalent relation generated by the following
\emph{basic} pairs:

\vspace{1ex}

%
% $(i\,\ell\,j\,s\,k) \sim (i\,t\,k)$ each time there is $w \in S_2$
%such that $\partial_0(w) = \ell$, \mbox{$\partial_1(w) = t$,}
%$\partial_2(w) = s$ 
%(the reader can check that the simplicial equations show that this
%fits correctly). 
%
 (1a) Given $i \mr{\ell} j \mr{r} k$, then $i \mr{(r\, \ell)} k \; \sim
\; i 
  \mr{t} k$ if there is $w \in S_2$ 
such that $\partial_2(w) = \ell$, $\partial_1(w) = t$, $\partial_0(w)
= r$.  That is, for each $w \in S_2$ we establish
$\partial_1(w) \sim (\partial_0(w) \, \partial_2(w))$.

The arrow $i \mr{\sigma_0(i)} i$ becomes the identity
morphism of $i$, $id_i = \sigma_0(i)$.

\vspace{1ex}

(1b) The groupoid is obtained by formally inverting all the arrows of the category.
 \cqd
\end{proposition}
%
%For later use we need the following definition, modeled after the
%category of sets with relations as morphisms.

%\vspace{1ex}

%Recall that the \emph{dual} or \emph{opposite} $S^{op}_\bullet$  of a
%simplicial set  $S_\bullet$ is the
%simplicial set that results after reversing the order of each
%ordinal.  Combinatorially it is obtained by the
%sustitution of $i$ by 
%$n-i$ in the faces, and $i$ by $n-1-i$ in the degenerancies. 
%Thus, $S^{op}_n = S_n$. 
%We distiguish a simplex $w \in S_n$ from the
%corresponding simplex of $S^{op}_n$ by writing $w^{op} \in
%S^{op}_n$. In this way we have 
%$\partial_i(w^{op}) = \partial_{n-i}(w)$, 
%$\sigma_i(w^{op}) = \sigma_{n-1-i}(w)$. 
%
\begin{definition}
A contravariant simplicial morphism $S_\bullet \mr{h_\bullet}
T_\bullet$ between two simplicial sets is a family of maps $S_n
\mr{h_n} T_n$ such that  
$$\partial_i(h_n(w)) = h_{n-1}\partial_{n-i}(w),\;\;\; 
\sigma_i(h_{n-1}(w)) = h_n(\sigma_{n-1-i}(w)).
$$   
\end{definition}
\begin{definition} \label{duality}
A \emph{strict duality} in a simplicial set $S_\bullet$ is a
contravariant simplicial 
isomorphism $S_\bullet \mr{\tau_\bullet} S_\bullet$ with $\tau_0 =
id$. We denote $\tau$
  in both directions 
$\tau \circ \tau^{-1} = id$.
%
%and that the  
%pairs \mbox{$i \mr{\ell} j 
%\mr{\tau(\ell)} i$}, $j \mr{\tau(\ell)} i 
%  \mr{\ell} j$ compose.  Plus higher order filling
%  conditions that we do not need to explicitate here.
%
 A simplicial set
  with a strict duality is said to be \emph{self-dual}.
\end{definition}

For any vertex $i$, $\tau_0(i) =
i$. For any $n$-simplex $w$, $n > 0$, we will
denote \mbox{$\tau_n(w) = w^{op}$,} omitting the $\,n\,$.
%
% If $i \mr{\ell} j$, then $j \mr{\ell^{op}} i$.
%
$$w \in S_n:\;\;\partial_i(w^{op}) = \partial_{n-i}(w)^{op}, \;\;\; 
\sigma_i(w^{op}) = (\sigma_{n-1-i}(w)^{op}.
$$
%A simplex $w$ is \emph{symmetric} if the $\tau$ on $w$ behaves as a
%morphism:
%$$w \in S_n:\;\;\partial_i(w^{op}) = \partial_i(w)^{op}, \;\;\; 
%\sigma_i(w^{op}) = \sigma_i(w)^{op}.
%$$
 
%The condition $\tau_0 = id$ can be replaced by an appropriate filling
%condition to define the notion of a (non strict) duality. For
%simplifying reasons we do not treat the general 
%case, which we do not need in this paper.  

\begin{remark}
Clearly, the notion of strict duality applies to a simplicial object
in any category.
\end{remark}
The following is clear:
\begin{proposition} \label{selfdualisgroupoid}
Let $S_\bullet$ be a self-dual simplicial set such that:
$$
\xymatrix@R=4.3ex
         { 
          {} \\ \forall \, i \mr{\ell} j \, \in S_1 \; \exists \, w \in S_2,
         } 
 \hspace{2ex} 
\xymatrix@R=1.5ex
         {
          {} & j \ar[rdd]^{\ell^{op}}
         \\
          {} & w
         \\
          i \ar[ruu]^\ell \ar[rr]^{\partial_1(w)} & {} & i \,,
         }
\hspace{2ex}
\xymatrix@R=5.3ex
         {
          {} \\  \partial_1(w) = \sigma_0(i).
         }
$$
Then the fundamental category is already  groupoid.
\cqd 
\end{proposition}
\begin{example} [Cech nerve] \label{cechnerve}
Given a family 
$\cc{U} = (U, I, \zeta)$, $\{U_i\}_{i \in I}$, $U \mr{\zeta} {\gamma}^{*}I$,
in a
topos $\cc{E} 
\mr{\gamma} \cc{S}$,
the \emph{Cech simplicial set} \footnote{Often called the
  \emph{nerve} of $\cc{U}$}
is the 
simplicial set $N_\bullet$ whose $n$-simplexes are given by
\mbox{$N_n = \{(i_0, i_1, \ldots i_n) \;|\;  U_{i_0} \times  U_{i_1} \ldots
\times U_{i_n} \;\neq\; \emptyset\} \subset I^{n+1}$}, 
in particular
$N_0 = I$, $N_1 = \{(i, \, j) \,|\, U_i \times U_j \neq \emptyset \}$, 
$N_2 = \{(i, \, j,\, k) \,|\, U_i \times U_j \times U_k \neq \emptyset
\}$.

The reader can check that it is a self-dual
simplicial set. For $i \in N_0$,  \mbox{$w =
(i,\,j,\,k) \in N_2$,} $\sigma_0(i) = (i,\,i)$, and  $\partial_2(w) = (i,\,j)$,  
$\partial_0(w) = (j,\,k)$, $\partial_1(w) = (i,\,k)$. Given
\mbox{$\ell = (i, \,j) \in N_1$,} $\ell^{op} = (j,\, i)$. Then $w =
(i,\,j,\,i)$ establish the condition in
proposition \ref{selfdualisgroupoid}. Thus the fundamental category is
a groupoid. 
\cqd  
\end{example}

\section{Simplicial families} \label{sf}

Recall that a family in a topos $\cc{E} \mr{\gamma} \cc{S}$ is an
arrow $\zeta \colon H \to {\gamma}^{*}S$. In alternative
notation we write $\cc{H} = \{H_{i}\}_{i \in S}$. We say that the
objects $H_i$ are the \emph{components} of $H$. Families are
\mbox{$3$-tuples} \mbox{$\cc{H} = (H, S, \zeta)$,} and  
$H$ is the coproduct \( \, H = \sum_{i \in S} H_i \, \) in $\cc{E}$.

\emph{Remark that the same 
object $H$ can be indexed by a different set, having then a different
set of components.}

A \emph{morphism of families} $(Y,\, J,
\,\xi) \mr{(h, \, \alpha)} (H, \,S, \,\zeta)$, is a pair $Y \mr{h} H$,
\mbox{$J \mr{\alpha} S$,}  
making the following square commutative:
$$
\xymatrix
        {
          Y \ar[r]^{h} \ar[d]^{\xi} 
        & H \ar[d]^{\zeta}
        \\
          \gamma^{*} J \ar[r]^{\gamma^*\alpha} 
        & \gamma^{*} S
        }
$$
In alternative notation, this
corresponds to  
$h = \{Y_{i} \mr{h_i} H_{\alpha(i)}\}_{i \in J}$.

We also say that $\cc{Y}$ is a \emph{refinement} of $\cc{H}$. 

\vspace{1ex}

{\bf Assumption}. \emph{We shall assume always that the components of the
families are non empty, $H_i \neq \emptyset$ for all $i \in I$.}

\vspace{1ex}

\begin{definition}
A \emph{simplicial family}  is a $3$-tuple 
$\,\cc{H}_\bullet = (H_\bullet, S_\bullet, \zeta_\bullet)$, where
$H_\bullet$, $S_\bullet$ are simplicial objects in $\cc{E}$,
$\cc{S}$ respectively, and $\,{H}_\bullet
\mr{\zeta_\bullet} \gamma^*({S}_\bullet)\,$ is a morphism of simplicial
objects in $\cc{E}$. Remark that we assume that $(H_n)_w \neq \emptyset$.

 A \emph{morphism of simplicial families}
$(Y_\bullet, J_\bullet, \zeta_\bullet) \mr{(h_\bullet, \,
  \alpha_\bullet)} (H_\bullet, S_\bullet, \zeta_\bullet)$, is a pair
$Y_\bullet \mr{h_\bullet} H_\bullet$, 
\mbox{$J_\bullet \mr{\alpha_\bullet} S_\bullet$,} of simplicial morphisms  
making the following square commutative:
$$
\xymatrix
        {
          Y_\bullet \ar[r]^{h_\bullet} \ar[d]^{\xi_\bullet} 
        & H_\bullet \ar[d]^{\zeta_\bullet}
        \\
          \gamma^{*} J_\bullet \ar[r]^{\gamma^*\alpha_\bullet} 
        & \gamma^{*} S_\bullet
        }
$$
We also say that $\cc{Y}_\bullet$ is a \emph{refinement} of
$\cc{H}_\bullet$.  

In alternative notation,  
$h_n = \{(Y_n)_{w} \mr{(h_n)_w} (H_n)_{\alpha_n(w)}\}_{w \in J_n}$.
\end{definition}

Notice that faces and degeneracy operators are
morphisms of families: 
$$
\xymatrix@C=12ex
         {
           H_n \ar@<1.2ex>[r]^{d_i} \ar[d]^{\zeta_n} 
         & H_{n-1} \ar@<1ex>[l]_{s_i} \ar[d]^{\zeta_{n-1}}
         \\
            \gamma^*{S}_n \ar@<1.3ex>[r]^{\gamma^*\partial_i}  
         &  \gamma^*{S}_{n-1} \ar@<1.3ex>[l]_{\gamma^*\sigma_i}
         }
$$
and that in alternative notation correspond to families of maps 
$$
\{\, (H_n)_w \mr{(d_i)_w} (H_{n-1})_{\partial_i(w)}\, \}_{w \in S_n}
\hspace{5ex} 
\{\, (H_n)_{\sigma_i(w)} \ml{(s_i)_w} (H_{n-1})_{w}\, \}_{w\in S_{n-1}}
$$

\vspace{5ex}

We make now some considerations involving the first three terms.

%and only the duality
%$\tau_1$ is used (recall that $\tau_0 = id$).
%
%\begin{equation}
$$
 \xymatrix@R=15ex@C=12ex 
   {
     {H}_2 \ar@<5.4ex>[r]^{d_0}  
           \ar[r]^{d_1}
           \ar@<-5.4ex>[r]^{d_2} 
           \ar[d]^{\zeta_{2}} 
   &  {H}_1  \ar@<2.8ex>[r]^{d_0}
             \ar@<-2.8ex>[l]_{s_0}
             \ar@<2.8ex>[l]_{s_1}   
             \ar@<-2.8ex>[r]^{d_1}
             \ar[d]^{\zeta_{1}}
   & {H}_0 \ar[l]_{s_0}  
           \ar[d]^{\zeta_0}
   \\ 
     \gamma^*{S}_2 \ar@<5.4ex>[r]^{\partial_0}  
                   \ar[r]^{\partial_1}
                   \ar@<-5.4ex>[r]^{\partial_2} 
   & \gamma^*{S}_1 \ar@<2.8ex>[r]^{\partial_0}
                   \ar@<-2.8ex>[l]_{\sigma_0}
                   \ar@<2.8ex>[l]_{\sigma_1}   
                   \ar@<-2.8ex>[r]^{\partial_1} 
   & \gamma^*{S}_0 \ar[l]_{\sigma_0}
  }
$$ 
%\end{equation}

%\vspace{1ex}

\begin{remark} \label{span} ${}$
\begin{enumerate}
\item \label{span1}
Each  1-simplex $\ell \in S_1$, $\,i \mr{\ell} j$ determines a 1-span 
$$
\xymatrix@C=3ex@R=3ex
         {
         {} &  (H_1)_\ell \ar[dl]_{(d_1)_\ell} \ar[dr]^{(d_0)_\ell}
         \\
         (H_0)_i & {} & (H_0)_j
         }
$$
%\mbox{$(H_0)_i \ml{(d_1)_\ell} (H_1)_\ell \mr{(d_0)_\ell}
%  (H_0)_j$.} 
\item \label{span2}
For each vertex $i \in S_0$ we have a morphism of spans:
$$
\xymatrix@C=10ex@R=1ex
         {
          {} & (H_0)_i \ar[dl]_{id} 
                          \ar[dr]^{id}
                          \ar[dd]^{(s_0)_i}
         \\
          (H_0)_i  & {} & (H_0)_i
         \\
          {} & (H_1)_{\sigma_{0}(i)} \ar[ul]^{(d_1)_{\sigma_0(i)}}
                                  \ar[ur]_{(d_0)_{\sigma_0(i)}}
         }
$$
\item \label{span3}
$
\xymatrix@R=5ex{\\ \textrm{Each 2-simplex} \; w \in S_2}
\hspace{2ex}
\xymatrix@R=1ex
         {
          {} & j \ar[rdd]^r
         \\
          {} & w
         \\
          i \ar[ruu]^\ell \ar[rr]^t & {} & k
         }
\hspace{1ex}
\xymatrix@R=5ex{\\ \textrm{determines a 2-span}}
$
$$
\xymatrix@C=4ex@R=5ex
         {
          {} & {} & (H_2)_w  \ar[d]^{(d_1)_w} 
                             \ar@/^1.5pc/[rdd]^{(d_0)_w} 
                             \ar@/_1.5pc/[ldd]_{(d_2)_w} 
         \\
          {} & {} & (H_1)_t  \ar@/^2pc/[rrdd]^{(d_0)_t} 
                             \ar@/_2pc/[lldd]_{(d_1)_t}
         \\
            {} & (H_1)_\ell  \ar[ld]^{(d_1)_\ell} 
                             \ar[rd]_{(d_0)_\ell} 
          & {} & (H_1)_r  \ar[ld]^{(d_1)_r} 
                          \ar[rd]_{(d_0)_r}
         \\
          (H_0)_i & {} &  (H_0)_j & {} & (H_0)_k\, 
         }
$$
$$
\hspace{-6ex}
\xymatrix@C=4ex@R=4ex
         {
     {} & {} & (H_2)_w \ar[dl]_{(p_2)_w} \ar[d]^{(p_1)_w} \ar[dr]^{(p_0)_w} & {}
         \\
          \textrm{The composites determine three maps} 
          \hspace{-3ex} 
         & (H_0)_i & (H_0)_j & (H_0)_k
         }
$$
\end{enumerate}
\cqd 
\end{remark}
\begin{definition} \label{sfgcG}
 A \emph{self-dual} simplicial family is a simplicial family
 \mbox{$\,\cc{H}_\bullet 
 = (H_\bullet, S_\bullet, \zeta_\bullet)$} together with a pair of strict self-dualities $\tau$ such that $\zeta_\bullet \circ
 \tau_\bullet = \gamma^*\tau_\bullet \circ \zeta_\bullet$. That is, 
 for $w \in S_n$,
 $(H_n)_w \mr{\tau_n} (H_n)_{w^{op}}$.
$$
\xymatrix@C=8ex
         {
            H_n \ar[r]^{d_{n-i}}  \ar[d]^{\tau_n} 
         &  H_{n-1} \ar[d]^{\tau_{n-1}}
         \\
            H_n \ar[r]^{d_i} 
         &  H_{n-1}
         }
\hspace{5ex}
\xymatrix@C=12ex
         {
           (H_n)_w  \ar[r]^{(d_{n-i})_w} \ar[d]^{(\tau_n)_w} 
         & (H_{n-1})_{\partial_{n-i}(w)} \ar[d]^{(\tau_{n-1})_{\partial_{n-i}(w)}}
         \\
           (H_n)_{w^{op}}  \ar[r]^{(d_i)_{w^{op}}} 
         & (H_{n-1})_{\partial_{n-i}(w)^{op}}
         }
$$
$\partial_{n-i}(w)^{op} = \partial_i(w^{op})$.
This means that
$\tau$ establishes an isomorphism between the span of $w^{op}$ and the
dual span of the span of $w$.
\end{definition}
%
%
%\vspace{1ex}
%
%\hspace{15ex} $\partial_i(w^{op}) = \partial_i(w)^{op}, \; 
%\sigma_i(w^{op}) = \sigma_i(w)^{op}$
%
%or 
%\hspace{12ex} 
%$\partial_{n-i}(w) = \partial_i(w)^{op}, \; 
%\sigma_{n-1-i}(w) = \sigma_i(w)^{op}
%$
%
%  
\begin{remark}\label{dualspan} ${}$
\begin{enumerate}
\item \label{dualspan1}
For $\ell \in S_1$, if $\,i \mr{\ell} j$, then $j \mr{\ell^{op}} i$, and 
$(H_1)_{\ell^{op}} \cong (H_1)_\ell^{op}$ in the sense that $(\tau_1)_\ell$ establishes an isomorphism between the span $\ell^{op}$ and the dual span of 
$\ell\;$:
$$
\xymatrix@C=10ex@R=1ex
         {
          {} & (H_1)_{\ell^{op}} \ar[dl]_{(d_1)_{\ell^{op}}} 
                          \ar[dr]^{(d_0)_{\ell^{op}}}
                          \ar[dd]^{(\tau_1)_\ell}_\cong
         \\
          (H_0)_j  & {} & (H_0)_i
         \\
          {} & (H_1)_{\ell} \ar[ul]^{(d_0)_\ell}
                               \ar[ur]_{(d_1)_\ell}
         }
$$
\item \label{dualspan2}

\vspace{-3ex}

$$
\xymatrix@R=4ex
         {
          \\ 
          \mathrm{For \; w \in S_2, \; if}
         }
\hspace{2ex}
\xymatrix@R=1ex
         {
          {} & j \ar[rdd]^r
         \\
          {} & w
         \\
          i \ar[ruu]^\ell \ar[rr]^t & {} & k
         }
\hspace{1ex}
\xymatrix@R=4ex
          {
           \\ 
           \mathrm{\;then\;\;\;}
          }
\xymatrix@R=1ex
         {
          {} & j \ar[rdd]^{\ell^{op}}
         \\
          {} & w^{op}
         \\
          k \ar[ruu]^{r^{op}} \ar[rr]^{t^{op}} & {} & i
         }
\xymatrix@R=4ex
         {
          \\ 
          \mathrm{\;\; and}
         }
$$
$(H_2)_{w^{op}} \cong  (H_2)_w^{op}$ in the sense that
that $(\tau_2)_w$ and $(\tau_1)_l, \, (\tau_1)_t, \, (\tau_1)_r$
establish an isomorphism between the 2-span of $w^{op}$ and the dual
2-span of $w$.  We leave to the interested reader to draw the
corresponding diagram .\cqd
\end{enumerate}
\end{remark}

\begin{example} [{Cech nerve simplicial family}] \label{cechfamily}
Consider a family \mbox{$U \mr{\zeta} {\gamma}^{*}I$}  in a topos $\,\cc{E}
\mr{\gamma} \cc{S}$, and the Cech
simplicial set  $N_\bullet$, see \mbox{example \ref{cechnerve}.}

 The \emph{canonical simplicial 
object} $U_\bullet$  is the simplicial object in $\cc{E}$ whose
 $n$-simplexes are given by 
\mbox{$U_n = U \times U  \times \cdots  U$ ($n+1$ times),} 
\mbox{\( \, U_n = \sum_{(i_0, i_1, \ldots i_n) \, \in \: N_n } 
                          \, U_{i_0} \times  U_{i_1} \ldots \times
                          U_{i_n} \, \),}
in particular 
$$
{U}_0 = \sum_{i \in \,N_0} U_i\,,\;\;\; 
{U}_1 = \sum_{(i,\,j) \in \,N_1} U_i \times U_j\,,\;\;\;
{U}_2 = \sum_{(i,\,j,\,k) \in \,N_2} U_i \times U_j \times U_k
$$
It is easy to see that $U_\bullet$ is a self-dual simplicial object
with faces given by the appropriate projections ans degeneracies by the
appropriate diagonals. As for the self-duality, 
$\tau_1$ is the usual symmetry of the cartesian product, $\tau_2$
permutes the first and third factors, and leave unchanged the middle
one, etc. The map $U \mr{\zeta} \gamma^{*}I$ determines a
self-dual simplicial family \mbox{$\, {U}_\bullet
  \mr{\zeta_\bullet}  
\gamma^*({N}_\bullet)\,$} that we call \emph{Cech simplicial family}.
\cqd  
\end{example}
\begin{proposition}\label{counit}
Given any simplicial family  $\,\cc{H}_\bullet = (H_\bullet,
S_\bullet, \zeta_\bullet)$, there is a canonical morphism of
simplicial families 
$$
\xymatrix
        {
          H_\bullet \ar[r]^{h_\bullet} \ar[d]^{\zeta_\bullet} 
        & U_\bullet \ar[d]^{\zeta_\bullet}
        \\
          \gamma^{*} S_\bullet \ar[r]^{\gamma^*\alpha_\bullet} 
        & \gamma^{*} N_\bullet
        }
$$
where $(U \mr{\zeta} \gamma^{*}I) = (H_0 \mr{\zeta_0} \gamma^{*}S_0)$.
If the family is self-dual, $h$ and $\alpha$ commute with the dualities.  
\end{proposition}
\begin{proof}
For the first three terms the proof should be clear by the remark
\ref{span}: Given $\ell \in S_1$ and $w \in S_2$, $\alpha_1(\ell) =
(\partial_1(\ell),\, \partial_0(\ell))$,  
$(h_1)_\ell = ((d_1)_\ell,\, (d_0)_\ell)$,  
$\alpha_2(w) = (\varrho_2(w),\, \varrho_1(w),\, \varrho_0(w))$, 
\mbox{$(h_2)_w = ((p_2)_w,\, (p_1)_w,\, (p_0)_w)$.} The second 
assertion follows by remark \ref{dualspan}. For the higher simplexes the proof
is the same and the 
interested reader can deduce the necessary calculations.
\end{proof}

\section{Family hypercovers}

We consider now $U \mr{\zeta} \gamma^{*}I$ to be a cover, that is the
map $U \to 
1$ an epimorphism, and we will establish the condition that says that
$H_\bullet 
\mr{h_\bullet} U_\bullet$ is a hypercover in the sense of
\cite{AM}.

\begin{sinnadastandard} [\bf{The coskeleton}] \label{thecoskeleton} 
\end{sinnadastandard}
%Recall that $U_\bullet = cosk_0(U_\bullet) = cosk_0(H_\bullet)$.
By construction of the coskeleton we have, for  $(\ell,\, t,\, r) \in
S_1 \times S_1 \times S_1$,     
$
\hspace{2ex}
\xymatrix@R=3.5ex{\\ (cosk_1 S_\bullet)_2 = \{(\ell,\, t,\, r) \; |  \;}
\xymatrix@R=1ex
         {
          {} & j \ar[rdd]^r
         \\
          {} & {}
         \\
          i \ar[ruu]^\ell \ar[rr]^t & {} & k
         }
\hspace{1ex}
\xymatrix@R=3.5ex{\\ \textrm{(no $w$ filling the triangle})\;\}}
$ 
$$
= \{(\ell,\, t,\, r) \;|\; 
\partial_0(\ell) = \partial_1(r) = j,\,
\partial_1(t) = \partial_1(\ell) = i,\,
\partial_0(t) = \partial_0(r) = k\,  \}
$$

\vspace{1ex}

Let $P_{ij}$ and $P_{\ell t r}$ be limit cones as follows:
$$
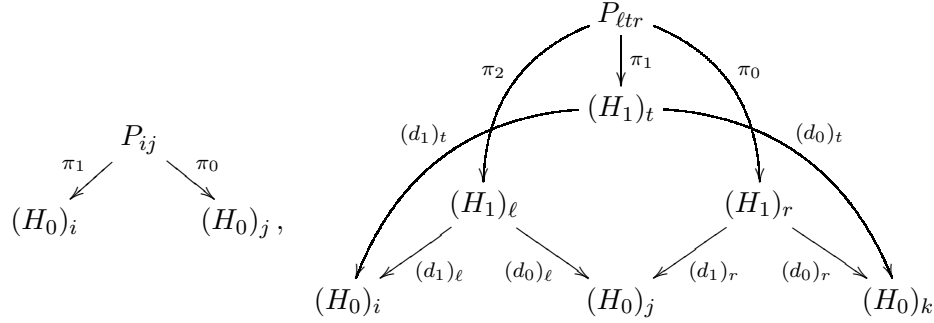
\hspace{2ex}
\xymatrix@C=2ex@R=3ex
         {
         \\
         \\
         {} & P_{ij} \ar[dl]_{\pi_1} \ar[dr]^{\pi_0} & {}
         \\
         (H_0)_i & {} & (H_0)_j\,, 
         }
\hspace{1ex}
\xymatrix@C=4ex@R=4ex
         {
          {} & {} & P_{\ell t r}  \ar[d]^{\pi_1} 
                             \ar@/^1.5pc/[rdd]^{\pi_0} 
                             \ar@/_1.5pc/[ldd]_{\pi_2} 
         \\
          {} & {} & (H_1)_t  \ar@/^2pc/[rrdd]^{(d_0)_t} 
                             \ar@/_2pc/[lldd]_{(d_1)_t}
         \\
            {} & (H_1)_\ell  \ar[ld]^{(d_1)_\ell} 
                             \ar[rd]_{(d_0)_\ell} 
          & {} & (H_1)_r  \ar[ld]^{(d_1)_r} 
                          \ar[rd]_{(d_0)_r}
         \\
          (H_0)_i & {} &  (H_0)_j & {} & (H_0)_k\, 
         }
$$
$P_{ij} = (H_0)_i \times (H_0)_j$, $N_1 = \{(i,\,j) \,| \, P_{ij} \neq
         \emptyset \,\} $. Let $T_2$ be  
$T_2 \subset (cosk_1 S_\bullet)_2$, 
\mbox{$T_2 = \{(\ell,\, t,\, r) \,| \, P_{\ell t r} \neq \emptyset
  \,\}$}. It follows $(i,\,j, \,k) \in N_2$, thus there is a function
$T_2 \mr{} N_2$. Clearly for $w \in S_2$,
$(\partial_2(w), \, \partial_1(w),\, \partial_0(w)) \in T_2$, defining
a function   
\mbox{$S_2 \mr{} T_2$.}
%\begin{sinnadastandard} \label{subobject}
%Remark there is a map $P_{\ell t r}  \mmr{}  (H_0)_i \times (H_0)_j
%\times (H_0)_k$, and it is not difficult to prove that it a
%monomorphism. 
%\end{sinnadastandard}
By construction of the coskeleton, we have:
 
\vspace{1ex}

\hspace{10ex} $(H_0)_i = U_i$, indexed by $S_0$, 

\vspace{1ex}

\hspace{10ex} $((cosk_0 H_\bullet)_1)_{ij} = P_{ij}$, indexed by $N_1$, 

\vspace{1ex}

\hspace{10ex} $((cosk_1 H_\bullet)_2)_{\ell t r} =  P_{\ell t r}$,
indexed by $T_2$.  
$$
H_0 = U_0, \hspace{5ex} (cosk_0 H_\bullet)_1 =  U_1 = \sum_{(i,\,j)
  \in \,N_1} P_{ij},
$$  
$$
(cosk_1 H_\bullet)_2 
\;\; = \;\sum_{(\ell,\,t,\,r) \in \,T_2} P_{\ell t r}
\;\; = \;\; \sum_{(i,\,j,\,k) \in \,N_2} \;\;
\sum_{(\ell,\,t,\,r) \in \,(T_2)_{ijk}} 
\;\;\; P_{\ell t r}\,.
$$
Clearly, for each $\ell \in S_1$, $i = \partial_1(\ell)$, $j =
\partial_0(\ell)$, 
there is a map $(H_1)_\ell \mr{} P_{ij}$, and for each $w \in S_2$,
$\ell = \partial_2(w)$, $t = \partial_1(w)$, $r = \partial_0(w)$, 
there is a map $(H_2)_w \mr{} P_{\ell t r}$, which are the components
of the maps $H_1 \mr{} (cosk_0 H_\bullet)_1$ and $H_2 \mr{} (cosk_1
H_\bullet)_2$. From these considerations it follows:
\begin{definition} \label{indexedhypercover}
 A indexed hypercover
refinement of a cover $(U \mr{\zeta} \gamma^{*}I)$ is a simplicial
family $H_\bullet \mr{\zeta} S_\bullet$, $S_0 = I,\; H_0 = U$ such
that the canonical morphism $H_\bullet \mr{} U_\bullet$ is a
hypercover, that is, the maps $H_{k} \mr{} (cosk_{k-1}
H_\bullet)_k$ are epimorphic (\cite{AM}). This is the case when for each
\mbox{$(i,\,j) \in N_1$} 
and $(\ell, t, r) \in T_2$, the families 
$\{(H_1)_\ell \mr{} P_{ij}\}_{\ell \in (S_1)_{ij}}$ and
$\{(H_2)_w \mr{} P_{\ell t r}\}_{w \in (S_2)_{\ell t r}}$
are epimorphic. \cqd
\end{definition}

As we have seen in remark \ref{span}, a simplicial family $H_\bullet
\mr{\zeta} S_\bullet$  
determines  a collection of spans in each dimension. The n-spans in
this collection are in one to one correspondence with the set $S_n$ of
n-simplexes 
of the index simplicial set, while the object $H_n$ of the simplicial
object is the coproduct of the vertices of all the n-spans indexed by $S_n$.
This collection of spans conform a set of data which is equivalent to
the simplicial family. Thus, we can determine a simplicial family by
specifying a suitable collections of spans.

\vspace{1ex}

Consider any family $U \mr{\zeta} {\gamma}^{*}I$. 
We will construct   
self-dual simplicial refinements of the Cech simplicial family
$$
\xymatrix
        {
          H_\bullet \ar[r]^{h_\bullet} \ar[d]^{\xi_\bullet} 
        & U_\bullet \ar[d]^{\zeta_\bullet}
        \\
          \gamma^{*} S_\bullet \ar[r]^{\gamma^*\alpha_\bullet} 
        & \gamma^{*} N_\bullet
        }
$$ 
determined by a given set of spans. This is similar to the
construction of the 
coskeleton functor.
\begin{construction} [0-Span refinements of the Cech simplicial
  family] \label{0spanref}
 \label{0hypercover}  ${}$

 Let $\cc{C}$ be a set of
\emph{non empty} objects closed under isomorphism and such that $U_i
\in \cc{C}$, all 
$i$. We will construct a self-dual simplicial refinement such that for
any $w \in S_n$, the components $(H_n)_w$ are in $\cc{C}$.  We
describe in detail the first three terms, where the 
  procedure is best understood. 

\vspace{1ex}

{\bf The simplicial set $S_\bullet\,$}:
\begin{enumerate}
\item
$S_0 = N_0$, $h_0 = id$.
\item
$S_1$ is the set of all 1-spans with vertex in $\cc{C}$ over the
objects $U_i$. We write $S_1 \mr{\alpha_1} N_1$, and for  $(i, \, j) \in
N_1$, define the fibers of $\alpha_1$ as:
$$
\xymatrix@C=3ex@R=1ex
         {
          {} & {} & V \ar[dl]_{u} \ar[dr]^{v} & {}
         \\
          (S_1)_{ij}\hspace{1ex} = \hspace{1ex} \{\,\ell \;=\;  \hspace{-4ex}
         & U_i & {} & U_j\,, 
         & \hspace{-3ex} \;  V \in \cc{C} \}
         }
$$ 
and  $\;\; \partial_0(\ell) = j$, $\partial_1(\ell) = i$, $\sigma_0(i) =
(U_i \ml{id} U_i \mr{id} U_i).$
\item
$S_2$ is the set of all 2-spans over the objects $U_i$ determined by
  the objects of $\cc{C}$. We write $S_2 \mr{\alpha_2} N_2$, and for
  $(i, \, j, \, k) \in 
N_2$, define the fibers of $\alpha_2$ as:

\mbox{
$
\hspace{-3.5ex}
\xymatrix@R=15.3ex
         {
          {} 
          \\
          (S_2)_{ijk}\hspace{1ex} = \hspace{1ex} \{w = 
         }
$
$
\xymatrix@C=4ex@R=3ex
         {
          {} & {} & W
          \ar[d]^{y} \ar@/^1.1pc/[rdd]^{z} \ar@/_1.1pc/[ldd]_{x} 
         \\
          {} & {} & Y  \ar@/^1.5pc/[rrdd]^{v_b} \ar@/_1.5pc/[lldd]_{u_b}
         \\
            {} & X \ar[ld]^{u_a} \ar[rd]_{v_a} 
          & {} & Z \ar[ld]^{u_c} \ar[rd]_{v_c}
         \\
          U_i & {} & U_j & {} & U_k\, 
         }
$
$
\hspace{-2ex}
\xymatrix@R=15.3ex
         {
          {} 
          \\
          , \;X,\; Y,\; Z,\; W \in \cc{C} \; \}
         }
$
}

\vspace{-3ex}

$$
\hspace{-3ex}
\xymatrix@C=3ex@R=3ex
         {
          {} & {} & W \ar[dl]_{f} \ar[d]^{h} \ar[dr]^{g} & {}
         \\
          \textrm{Taking the respective composites we have three maps} 
          \hspace{-3ex} 
         & U_i & U_j & U_k\,.
         }
$$
The face operators are clear: 
%$\partial_2(w) = \ell$, $\partial_1(w)
%= t$, $\partial_0(w) = r$.
$$
\hspace{-2.5ex}
\xymatrix@C=1.6ex@R=1ex
         {
          {} & {} & X \ar[dl]_{u_a} \ar[dr]^{v_a}
         \\
         \partial_2(w)   \;=\;  \hspace{-3ex}
         & U_i & {} & U_j\,, 
         }
\xymatrix@C=1.6ex@R=1ex
         {
          {} & {} & Y \ar[dl]_{u_b} \ar[dr]^{v_b}
         \\
           \partial_1(w) \;=\;  \hspace{-3ex}
         & U_i & {} & U_k\,, 
         }
\xymatrix@C=1.6ex@R=1ex
         {
          {} & {} & Z \ar[dl]_{u_c} \ar[dr]^{v_c}
         \\
           \partial_0(w) \;=\;  \hspace{-3ex}
         & U_j & {} & U_k 
         }
$$
Since $W \neq \emptyset$,
$(\partial_2(w),\, \partial_1(w), \, \partial_0(w)) \in T_2$ (see
\ref{thecoskeleton}),  $S_2 \mr{\alpha_2} N_2$ factors $S_2 \mr{} T_2
\mr{} N_2$. 

Given $\ell \in S_1$, the degeneracy operators 
$\sigma_0(\ell)$ and  $\sigma_1(\ell)$ are given by:
$$
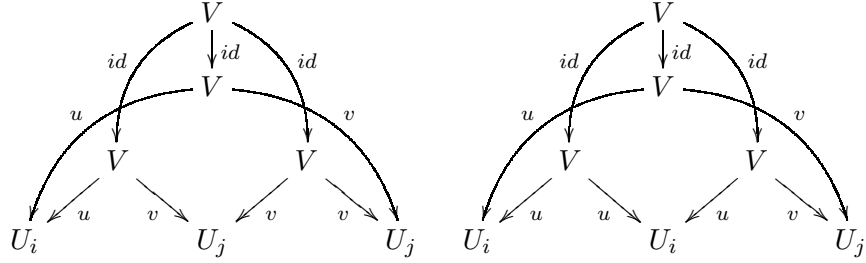
\xymatrix@C=4ex@R=3ex
         {
          {} & {} & V
          \ar[d]^{id} \ar@/^1.1pc/[rdd]^{id} \ar@/_1.1pc/[ldd]_{id} 
         \\
          {} & {} & V  \ar@/^1.5pc/[rrdd]^{v} \ar@/_1.5pc/[lldd]_{u}
         \\
            {} & V \ar[ld]^{u} \ar[rd]_{v} 
          & {} & V \ar[ld]^{v} \ar[rd]_{v}
         \\
          U_i & {} & U_j & {} & U_j\, 
         }
\;\;\;
\xymatrix@C=4ex@R=3ex
         {
          {} & {} & V
          \ar[d]^{id} \ar@/^1.1pc/[rdd]^{id} \ar@/_1.1pc/[ldd]_{id} 
         \\
          {} & {} & V  \ar@/^1.5pc/[rrdd]^{v} \ar@/_1.5pc/[lldd]_{u}
         \\
            {} & V \ar[ld]^{u} \ar[rd]_{u} 
          & {} & V \ar[ld]^{u} \ar[rd]_{v}
         \\
          U_i & {} & U_i & {} & U_j\, 
         }
$$
\end{enumerate}
{\bf The simplicial object $H_\bullet\,$}:

It is determined by talking the coproducts of the vertices of the
spans:
\begin{enumerate}
\item
$H_0 \mr{\zeta_0} {\gamma}^{*}S_0$ is defined by:
$$
(H_0)_i = U_i\, , \; H_0 = \sum_{i \in \,S_0} U_i\,
$$

\vspace{-2ex}

\item
$H_1 \mr{\xi_1} S_1$ is  defined by:
$$
(H_1)_{\ell} = V, \hspace{5ex} {H}_1 \; = \sum_{\ell\, \in \,S_1} V \;\; =
\;\sum_{(i,\,j) \in \,N_1} 
\;\;\sum_{\ell\, \in \,(S_1)_{ij}} V.
$$
Then $(s_0)_i = id_{U_i}$, 
$(d_0)_{\ell} = v$, and $(d_1)_{\ell} = u$. The map 
$V \mr{(u,\,v)} U_i \times U_j$ induce a map
$H_1 \mr{h_1} U_1$ commuting with the simplicial structure.
\item 
$H_2 \mr{\xi_2} S_2$ is  defined by:
$$
(H_2)_{w} = W\,, \hspace{5ex} {H}_2 
\; = \sum_{w\, \in \,S_2} W \;\; = \;\sum_{(\ell,\,t,\,r) \in \,T_2}
\;\;\sum_{w\, \in \,(S_2)_{\ell t r}} W
$$
$$
\hspace{-30ex} = \;\; \sum_{(i,\,j,\,k) \in \,N_2} \;\;
\sum_{(\ell,\,t,\,r) \in \,(T_2)_{ijk}} 
\;\;\; \sum_{w\, \in \,(S_2)_{\ell t r}} W\,.
$$

Define $(s_0)_\ell = id_V$, $(s_1)_\ell = id_V$, $(d_0)_w = z$,
$(d_1)_w = y$, $(d_2)_w = x$.
The map $W \mr{(f,\,g,\,h)} U_i \times U_j \times U_k$  induce a map
$H_2 \mr{h_2} U_2$ commuting with the simplicial structure.
\end{enumerate}

%       VENIR A BUSCAR      %
%
%Clearly there is an
%induced  map
%$H_2 \mr{h_2} U_2$ commuting with the simplicial structure. By the
%construction of the coskeleton  and the 
%fact that the connected objects $W \in \cc{C}$ generate, it follows
%that there is an epimorphic map $H_2 \mr{} (cosk_1C_\bullet)_2$
%commuting with the simplicial structure, so that we have an hypercover.
%
%\vspace{1ex}
%

{\bf The self-duality $\tau\,$:}

 Recall that $\tau_0 = id$. 
%
%The dual 2-span of the 2-span $w$ in diagram
%\ref{coskeleton} will be a 2-span isomorphic to the 2-span
% $w^{op} = \tau_2(w)$.
% 
Given $\ell \in S_1$, we define $\tau_1(\ell) = \ell^{op}$ as the
pullback on the right below: 
$$
\xymatrix@C=1.5ex@R=3ex
         {
          {} & {} & V^{op} \ar[dl]_{u^{op}} \ar[dr]^{v^{op}} & {}
         \\
           \ell^{op} \;=\;\;\;  \hspace{-3ex}
         & U_j & {} & U_i\,,
         }
\hspace{8ex}
\xymatrix@C=10ex
         {
           V^{op} \ar[r]^{\hspace{-2ex} (u^{op},\,v^{op})}  \ar[d]^{(\tau_1)_\ell}
         & U_j \times U_i  \ar[d]^{\tau_1}
         \\
           V \ar[r]^{\hspace{-3ex} (u,\,v)} 
         & U_i \times U_j
         }
$$
It follows $u^{op} = v\circ(\tau_1)_\ell$ and $v^{op} =
u\circ(\tau_1)_\ell$, so that $(\tau_1)_\ell$ establishes an
isomorphism between the span $\ell^{op}$  and the dual span of
$\ell$. 
This shows that $\tau_1$ commutes with the family structure,
\mbox{$\xi_1 \circ \tau_1 = \gamma^*(\tau_1) \circ \xi_1$} (see \ref{sfgcG}). 

\vspace{1ex}

Given $w \in S_2$, we define $\tau_2(w) = w^{op}$ as the following 2-span:
$$
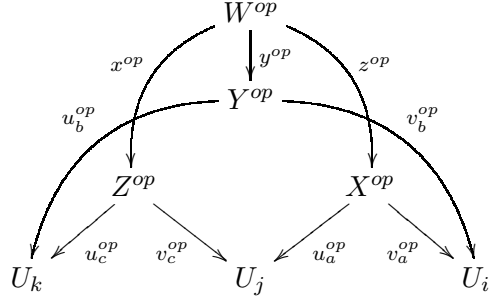
\xymatrix@C=4ex@R=4ex
         {
          {} & {} & W^{op}
          \ar[d]^{y^{op}}  \ar@/^1.5pc/[rdd]^{z^{op}} 
                           \ar@/_1.5pc/[ldd]_{x^{op}} 
         \\
          {} & {} & Y^{op}  \ar@/^2pc/[rrdd]^{v_b^{op}} 
                            \ar@/_2pc/[lldd]_{u_b^{op}}
         \\
            {} & Z^{op} \ar[ld]^{u_c^{op}} \ar[rd]_{v_c^{op}} 
          & {} & X^{op} \ar[ld]^{u_a^{op}} \ar[rd]_{v_a^{op}}
         \\
          U_k & {} & U_j & {} & U_i\, 
         }
$$
Where $W^{op}$ is defined as the pullback on the right below:
$$
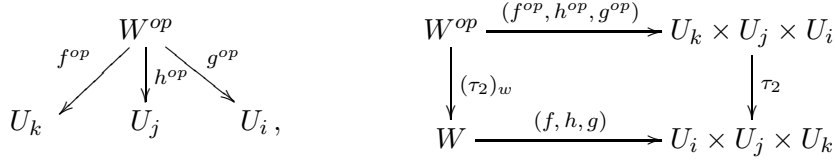
\xymatrix@C=4ex@R=4ex
         {
          {} & W^{op} \ar[dl]_{f^{op}} \ar[d]^{h^{op}} \ar[dr]^{g^{op}} & {}
         \\
          \hspace{-4ex}
          U_k & U_j & U_i\,, 
         }
         \hspace{10ex} 
\xymatrix@C=14ex
         {
           W^{op} \ar[r]^{\hspace{-5ex} (f^{op},\,h^{op},\,g^{op})}  
                  \ar[d]^{(\tau_2)_w}
         & U_k \times U_j \times U_i  \ar[d]^{\tau_2}
         \\
           W \ar[r]^{\hspace{-5ex} (f,\,h,\,g)} 
         & U_i \times U_j \times U_k
         }
$$ 
We have  $f^{op} = g\circ(\tau_2)_w$, $h^{op} = h\circ(\tau_2)_w$ and 
$ g^{op} = f\circ(\tau_2)_w$. It follows there are maps
$x^{op},\,y^{op},\,z^{op}$ (as indicated in the span diagram above)
which satisfy the equations: 
 
\vspace{-4ex}
 
$$
(\tau_1)_r \circ x^{op} = z \circ (\tau_2)_w, \; \;
(\tau_1)_t \circ y^{op} = y \circ (\tau_2)_w \;\; and \;\;  
(\tau_1)_\ell \circ z^{op} = x \circ (\tau_2)_w \,.
$$
This shows that $(\tau_2)_w$ and $(\tau_1)_\ell,\, (\tau_1)_t ,\,
(\tau_1)_r$ establish an isomorphism between the 2-span  
$w^{op}$  and the dual 2-span of $w$. That is, $\tau_2$ commutes with
the family structure,  
\mbox{$\xi_2 \circ \tau_2 = \gamma^*(\tau_2) \circ \xi_2$} (see
\ref{sfgcG}).
\cqd \end{construction}
   
\begin{construction}
[1-Span refinements of the Cech simplicial
  family] \label{1hypercover} \label{1spanref} ${}$

Let $\cc{C}_{sp}$ be a set of \emph{non empty} 1-spans closed under
isomorphisms, the dual span, such that $(U_i \ml{id} U_i \mr{id}
U_i) \in \cc{C}_{sp}$, all $i$, and such that \mbox{$U_i \ml{u} V \mr{u}
U_i\,$,} $\,U_j \ml{v} V \mr{v} U_j$ $\in \cc{C}_{sp}$ for all $U_i \ml{u}
V \mr{v} U_j$ $\in \cc{C}_{sp}$.  Let $\cc{C}$ be the set of vertices
of the spans in  $\cc{C}_{sp}$. 
We will construct a self-dual
simplicial refinement such that the set of  1-spans determined by the
1-simplexes is the set $\cc{C}_{sp}$ (and for 
$w \in S_n$, the component $(H_n)_w$ is in $\cc{C}$).

\vspace{1ex}

With the notation in construction \ref{0hypercover}, the $0$-term is
the same than 
in \ref{0hypercover}. The set $S_1$ is just defined to be the set
$\cc{C}_{sp}$ with the same simplicial structure, and $S_2$ is also
defined in the same way, but with the assumption that the three 1-spans
with vertices $X,\, Y,\, Z$ should be in $\cc{C}_{sp}$. 
The simplicial object $H_\bullet$ is defined exactly as in
\ref{0hypercover}, and from the fact that $\cc{C}_{sp}$ is closed
under the dual span and isomorphisms it follows that the definition of
the selfduality
$\tau$ in \ref{0hypercover} also applies here.
\cqd \end{construction}
  
%$$(H_0 \mr{\xi_0} S_0) = (U_0 \mr{\xi_0} N_0) \;\; \textrm{and} \;\; h_0 = id,
%\; \alpha_0 = id. 
%$$ 
%Write
%$$  
%$$
%(S_1)_{ij} = \{\ell = (U_i \ml{u} V \mr{v} U_j),\; \emptyset \neq V
%         \in \cc{C} \} 
%$$
%
%    VENIR A BUSCAR    %
%where $\cc{C} \in \cc{S}$  is the set of all connected
%objects 
%\footnote{Recall that this is a small set. Also, we can take
%  $\cc{C}$ to be any appropriate set of connected generators.}.
%
%Remark
%that $\ell \in S_1$ carries and it is determined by the data 
%in the $5$-tuple $\ell = (i,\,j,\,V,\,u,\,v)$. 
%
%Then, $\partial_0(\ell) = j$, $\partial_1(\ell) = i$, $\sigma_0(i) =
%(U_i \ml{id} U_i \mr{id} U_i).$
%$ = (i,\,i,\, U_i,\, id,\,id)$.
%   VENIR A BUSCAR %
% and which
%is epimorphic since $\cc{C}$ is a set of generators. 
%
%
%\vspace{1ex}
%
% $H_2 \mr{\xi_2} S_2$ is determined by requiring that $H_2 \mr{}
%(cosk_1C_\bullet)_2$ be a cover by connected objects, so that
%$H_\bullet$ is an hypercover in the sense of \cite{AM}.
%
% VENIR A BUSCAR %
%
%Furthermore, $H_\bullet$ will be an hypercover in the sense of  \cite{AM}.
%
From definition \ref{indexedhypercover} we have:
\begin{proposition} \label{hypercovering}
Given a cover $U \mr{\zeta} {\gamma}^{*}I$, if for each $(i,\,j) \in
N_1$ and $(\ell,\, t,\, r) \in T_2$ we have:
\begin{enumerate}
\item \label{0indexedhypercover}
If $\cc{C}$ as in construction \ref{0spanref} is such that  the families of all
maps \mbox{$\{W \mr{} P_{ij}\}_{W \in 
  \cc{C}}$} and $\{W \mr{} P_{\ell t r}\}_{W \in \cc{C}}$ are
epimorphic. \hspace{3ex} Or
\item \label{1indexedhypercover}
If $\cc{C}_{sp}$ and  $\cc{C}$ as in construction \ref{1spanref} are
such that  the 
families of all 
maps \mbox{$\{W \mr{(u, \, v)} P_{ij}\}_{(u,\,v) \in \cc{C}_{sp}}$} and  $\{W
\mr{} P_{\ell t r}\}_{W \in \cc{C}}$  are
epimorphic. 
\end{enumerate}
Then the span refinement $H_\bullet \mr{h_\bullet} U_\bullet$ is a hypercover.
\cqd 
\end{proposition}
\begin{example} [Canonical hypercover refinement of the Cech simplicial
      family by connected objects]  \label{connectedhypercover} ${}$

We assume the topos (or the site) to be \emph{locally connected}
(\cite{AM}, \cite{MM}). Consider a cover $\cc{U} = (U, S, \zeta)$ ,
$U \mr{\zeta} {\gamma}^{*}I$ such that all the $U_i$ are connected
objects. Then, taking as $\cc{C}$ any set of connected generators the
\mbox{construction \ref{0hypercover}} yields an hypercover refinement of the
Cech simplicial family in which all the components are connected.  
\cqd\end{example}

\section{Fundamental groupoid of a simplicial family}

We will now associate a groupoid in $\cc{S}$ to any self-dual
simplicial family satisfying the following filling condition. We
remark that this condition does not hold for the Cech simplicial
family but it will hold for the span refinements.
%
%\ref{cechfamily}. and \ref{connectedcomponents}.
%
\begin{definition} \label{conditionG}
Let $H_\bullet \mr{\xi}
\gamma^*(S_\bullet)$, $\tau$, be any
 self-dual simplicial family, we say that \emph{condition $G$} is
 satisfied if the following holds: 
$$
\xymatrix@R=5ex
         { 
          {} \\ \textrm{For all} \; \; i \mr{\ell} j \, \in S_1 \; \textrm{there exists a 2-simplex} \; w \in S_2,
         } 
\hspace{2ex} 
\xymatrix@R=1.5ex
         {
          {} & j \ar[rdd]^{\ell^{op}}
         \\
          {} & w
         \\
          i \ar[ruu]^{\ell} \ar[rr]^{\partial_1(w)} & {} & i \,,
         }
$$
such that $(d_1)_{\partial_1(w)} =(d_0)_{\partial_1(w)}\,$.
\end{definition}
\begin{proposition} \label{constructionG}
Any $0$-span or $1$-span simplicial refinement \mbox{$H_\bullet \mr{\xi}
\gamma^*(S_\bullet)$}
of a family $U \mr{\zeta} {\gamma}^{*}I$  as in constructions
\ref{0hypercover} or \ref{1hypercover} satisfies condition G.
\end{proposition} 
\begin{proof}
Let $\ell \in S_1$ be any 1-simplex, and let $w \in S_2$ be as follows: 
$$ 
\xymatrix@C=3ex@R=3ex
         {
         \\
          {} & {} & V \ar[dl]_{u} \ar[dr]^{v} & {}
         \\
          \ell \;=\;  \hspace{-4ex}
         & U_i & {} & U_j\,, 
         & {}
         }
\xymatrix@C=4ex@R=3ex
         {
          {} & {} & {} & V
          \ar[d]^{id} \ar@/^1.1pc/[rdd]^{(\tau_1)_\ell} \ar@/_1.1pc/[ldd]_{id} 
         \\
          {} & {} & {} & V  \ar@/^1.5pc/[rrdd]^{u} \ar@/_1.5pc/[lldd]_{u}
         \\
           {} & {} & V \ar[ld]^{u} \ar[rd]_{v} 
          & {} & V^{op} \ar[ld]^{u^{op}} \ar[rd]_{v^{op}}
         \\
          w \;=\;  \hspace{-4ex}
          & U_i & {} & U_j & {} & U_i\, 
         }
$$
It is clear that $w$ meets the requirements of condition G. 
\end{proof}

% In the following  
%construction only the three first terms are concerned. 
\begin{proposition} [G-Fundamental Groupoid of a simplicial
    family] \label{groupoid} ${}$

Let $H_\bullet \mr{\xi}
\gamma^*(S_\bullet)$, $\tau$, be a self dual simplicial family satisfying condition
G.
%
%in a topos \mbox{$\cc{E} \mr{\gamma} \cc{S}$.}
%
%For all $i \in S_0$, $(H_0)_i \mr{(s_0)_i} (H_1)_{\sigma_0(i)}$ is
%an isomorphism. 
%
Consider the fundamental category of the
simplicial set $S_\bullet$ \mbox{(proposition
\ref{fundamentalcat}).} Add to the equivalence relation that
  defines the morphism the following pairs:

\vspace{1ex}

(2)  $(i \mr{\ell} j) \sim  (i \mr{t} j)$ if there is a morphism of spans:
$$
\xymatrix@C=10ex@R=1ex
         {
          {} & (H_1)_\ell \ar[dl]_{(d_0)_\ell} 
                          \ar[dr]^{(d_1)_\ell}
                          \ar[dd]
         \\
          (H_0)_i  & {} & (H_0)_j
         \\
          {} & (H_1)_{t} \ar[ul]^{(d_0)_{t}}
                         \ar[ur]_{(d_1)_{t}}
         }
$$
Then, the resulting category is groupoid.
\end{proposition}
\begin{proof}
We have:
$$
\xymatrix@C=10ex@R=1ex
         {
          {} & (H_1)_{\partial_1(w)} \ar[dd]
          \ar[ddl]_{(d_1)_{\partial_1(w)}}
          \ar[ddr]^{(d_0)_{\partial_1(w)}} 
         \\
         \\
          (H_0)_i & (H_0)_i \ar[l]_{id} 
                          \ar[r]^{id}
                          \ar[dd]^{(s_0)_i} 
          & (H_0)_i
         \\
           {}  & {} & {}
         \\
          {} & (H_1)_{\sigma_0(i)} \ar[uul]^{(d_1)_{\sigma_0(i)}}
                                  \ar[uur]_{(d_0)_{\sigma_0(i)}}
         }
$$
This shows that $(\ell^{op}\, \ell)
\sim \partial_1(w) \sim \sigma_0(i)$. We use the assumption in the 1-simplex
$\ell^{op}$ to show $(\ell \, \ell^{op}) \sim \sigma_0(j)$. 
\end{proof}
\emph{The G-fundamental groupoid of a family does not coincide with
  the fundamental groupoid of the index simplicial set}. Because of
this we have added the letter $G$ to the word $fundamental$ for the
groupoid constructed in proposition \ref{groupoid}.  

\section{Descent}

Let $S_\bullet$ be a simplicial set, $S_0 = I$, 
\begin{definition} \label{descent}
A $S_\bullet$-descent datum on a family indexed by $I$, $R \mr{} I$,
is an isomorphism in $\cc{S}_{/I}$, and it consists of the following
data: 

\hspace{-4ex} For each $1$-simplex $i \mr{\ell} j$ a bijection $R_i
\mr{s_\ell} R_j$ such that: 

\vspace{1ex}

1) For each vertex $i \in I = S_0$, \hspace{2ex} $s_{\sigma_0(i)} = id_{R_i}$.

\vspace{1ex}

2) For each $2$-simplex $w \in S_2$,  \hspace{2ex}
 $s_{\partial_1(w)} = s_{\partial_0(w)} \circ s_{\partial_2(w)}$.
%
%2) For each $2$-simplex $w \in S_2$,  \hspace{2ex} $s_t = s_r \circ s_\ell$.
%
%\hspace{30ex} 
%(where $\partial_2(w) = \ell$, $\partial_1(w) = t$, $\partial_0(w) = r$)
\end{definition}
Recall that a \emph{(left) action} of a small category with set of
objects $I$ in a $I$-indexed family $R \mr{} I$ is a (covariant)
set-valued functor $R$, $R(i) = R_i$, and for $x \in R_i$, $\ell
\!\centerdot \! x = R(\ell)(x)$. We say that the action is \emph{by
  isomorphisms} if $R(\ell)$ is a bijection for all $\ell$. Actions by
isomorphisms are the same thing that actions of the groupoid resulting
by formally inverting all the arrows of the category. The following is
straightforward. For the record:  
\begin{proposition} \label{descent=action}
For any simplicial set $S_\bullet$, there is a one to one
correspondence between $S_\bullet$-descent data and left actions by
isomorphisms of the fundamental category, that is, actions of the
fundamental groupoid. With the obvious definition of morphisms this
bijection extends to an isomorphism of categories (and of topoi). 
\cqd\end{proposition}
Let $H_\bullet \mr{} S_\bullet$ be any simplicial family, $S_0 = I$, $H_0 = U$,
\begin{definition} \label{descentf}
A $H_\bullet$-descent datum $\sigma$ on an object 
$\gamma^*(R) \, \times_{\gamma^*(I)} \, U \mr{} U$ (where $R \mr{} I$
is a family indexed by I), is an isomorphism in $\cc{E}_{/U}$, and it
consists of the following data:  

For each $1$-simplex $i \mr{\ell} j$ an isomorphism $\sigma_\ell$:
$$
\xymatrix
         {
            \gamma^{*}R_i \times (H_1)_\ell  \ar[r]^{\sigma_{\ell}}
                                                   \ar[d]
          & \gamma^{*}R_j \times (H_1)_{\ell^{op}} \ar[d] 
          \\
            (H_1)\ell  \ar[r]^{(\tau_1)_\ell}  
          & (H_1)_{\ell^{op}}
         }
$$
Notice that $\sigma_\ell$ is completely determined by its first projection 
 that
we denote with the same letter \mbox{$\gamma^{*}R_i \times (H_1)_\ell
  \mr{\sigma_\ell} \gamma^{*}R_j$}, or  $(H_1)_\ell \mr{\widehat{\sigma}_\ell}
\gamma^{*}R_j^{\textstyle \,\gamma^{*}R_i}$. The identity and cocycle conditions 
are: 

\vspace{1ex}

 (1) For each $i \in S_0$,
$\; y \in \gamma^{*}R_i$, $x \in (H_0)_i\,$:  

\vspace{-1ex}

\begin{center}
$\; \sigma_{\sigma_0(i)}(y,\,(s_0)_i(x)) \;=\; (y,\,(s_0)_i(x))\,$:
$\; \; \widehat{\sigma}_{\sigma_0(i)}((s_0)_i(x)) \;=\; id_{\gamma^{*}R_i}$,
\end{center}

\vspace{1ex}

 (2) For each $w \in S_2$, 
     $\; y \in \gamma^{*}R_i$, $ x \in (H_2)_w\,$:  

\vspace{-1ex}

\begin{center} 
$\; \sigma_{\partial_0(w)}(\sigma_{\partial_2(w)}(y,\, (d_2)_w(x)),\,
  (d_0)_w(x)) \;=\; 
\sigma_{\partial_1(w)}(y,\, (d_1)_w(x))\,$: \\ \hspace{2ex} 
$\; \widehat{\sigma}_{\partial_0(w)}(d_0(x)) \circ
\widehat{\sigma}_{\partial_2(w)}(d_2(x)) \;=\;  
\widehat{\sigma}_{\partial_1(w)}(d_1(x))$,
\end{center}
\end{definition}
The equations above correspond to commutative diagrams, the
letters $x, \, y$ can be thought as internal variables, or simply
as a way to indicate how to construct the diagram.

Recall now from \cite[10.3]{AM}.
\begin{proposition} \label{hyper=all}
Given a cover $U \mr{\zeta} \gamma^*I$, a simplicial hypercover
refinement of the Cech simplicial family (see proposition \ref{counit}).
$$
\xymatrix
        {
          H_\bullet \ar[r]^{h_\bullet} \ar[d]^{\zeta_\bullet} 
        & U_\bullet \ar[d]^{\zeta_\bullet}
        \\
          \gamma^{*} S_\bullet \ar[r]^{\gamma^*\alpha_\bullet} 
        & \gamma^{*} N_\bullet
        }
$$
and a family of sets $R \mr{} I$ indexed by I, consider  
for all $i \mr{\ell} j \in S_1$, $(i,\,
j) = \alpha_1(\ell)$, the following diagrams:
$$
\xymatrix
         {
            \gamma^{*}R_i \times U_i \times U_j  \ar[r]^{\sigma_{j,\,i}}
          & \gamma^{*}R_j \times U_j \times U_i 
          \\
            \gamma^{*}R_i \times (H_1)_\ell \ar[r]^{\sigma_\ell}
            \ar[u]_{\gamma^{*}R_i \times (h_1)_\ell}  
          & \gamma^{*}R_j \times (H_1)_{\ell^{op}}
            \ar[u]_{\gamma^{*}R_j \times (h_1)_{\ell^{op}}}  
         }
\hspace{5ex}
\xymatrix
         {
          U_i \times U_j  \ar[dr]^{\widehat{\sigma}_{j,\,i}}
          \\
          (H_1)_\ell \ar[u]^{(h_1)_\ell}
                     \ar[r]^{\widehat{\sigma}_\ell}
          & \gamma^{*}R_j^{\textstyle \,\gamma^{*}R_i}
         }
$$
Then, composing with $(h_1)_\ell$, that is the correspondence
\mbox{$\sigma_{j,\,i}  \;\mapsto\; \sigma_\ell$}, where $\sigma_\ell$
    is such that  
\mbox{$\sigma_{j,\,i} \circ \gamma^{*}R_i \times (h_1)_\ell =
\gamma^{*}R_j \times (h_1)_{\ell^{op}} \circ \sigma_\ell\;$} or     
\mbox{$\;\widehat{\sigma}_{j,\,i} \;\mapsto\; 
\widehat{\sigma}_\ell = \widehat{\sigma}_{j,\,i} \circ (h_1)_\ell$,}
induces a bijection between 
$H_\bullet$-descent data 
$\sigma_\ell$  and $U_\bullet$-descent data $\sigma_{j,\,i}$  on
objects of the form   $\gamma^*(R) \, 
\times_{\gamma^*(I)} \, U \mr{} U$. This actually establishes an isomorphism of
the respective categories. 
\end{proposition}
\begin{proof}
We give only an sketch of the proof.  The correspondence is
injective since the family $\{(h_1)_\ell \}_{\ell \in (S_1)_{ij}}$ is
epimorphic. Given a $H_\bullet$-descent datum $\sigma_\ell$, it can be seen that
the family $\{\widehat{\sigma}_\ell 
\}_{\ell \in (S_1)_{ij}}$ is compatible, thus there exists a unique
$\widehat{\sigma}_{j,\,i}$ such that $\widehat{\sigma}_\ell =
\widehat{\sigma}_{j,\,i} \circ (h_1)_\ell$ all $\ell \in
(S_1)_{ij}$. The cocycle and identity equations follow
the fact that the
family (remark \ref{span}, \ref{span3}.) 
\mbox{
$\xymatrix@C=18ex
          {
           \{(H_2)_w \ar[r]^{\hspace{-8ex}((p_2)_w,\,(p_1)_w,\,(p_0)_w)} 
          & U_i \times U_j \times U_k\}_{w \in (S_2)_{ijk}} 
          }
$
} 
is epimorphic.   
\end{proof}
\begin{remark} \label{hyper=allremark}
A $S_\bullet$ descent datum $s_\ell$ on a family $R \mr{} I$ induces a
\mbox{$H_\bullet$-descent} datum $\sigma_\ell = \gamma^*(s_\ell)
\times (\tau_1)_\ell$ on the object $\gamma^*(R) \times_{\gamma^*(I)}
U$.  
\cqd
\end{remark}
\begin{sinnadastandard} \label{consistent}
We say that a $S_\bullet$-descent datum (as in definition \ref{descent}) is 
\emph{consistent} if for each pair of $1$-simplexes $\ell,\, t \in
S_1$ as in (2) proposition \ref{groupoid},  $s_\ell = s_t$.
\end{sinnadastandard}
We have
(compare with proposition \ref{descent=action}): 
\begin{proposition} \label{descent=actionG}
Given any self dual simplicial family $H_\bullet \mr{\xi} 
\gamma^*(S_\bullet)$ satisfying condition
G, there is a one to one correspondence between
consistent \mbox{$S_\bullet$-descent} data and left actions of the
G-fundamental groupoid of the family (Proposition \ref{groupoid}). With
the obvious definition of morphisms this bijection extends to an
isomorphism of categories (and of topoi). 
\cqd\end{proposition}
 
\section{Covering projections}

We recall now the concept of \emph{covering projection} introduced in
\cite{D3}, for details we refer the reader to this source. Consider a
cover  
\mbox{$\cc{U} = U \mr{\zeta} {\gamma}^{*}I$}  in a topos $\,\cc{E}
\mr{\gamma} \cc{S}$, and the Cech simplicial family    
$\, {U}_\bullet \mr{\zeta_\bullet}  \gamma^*({N}_\bullet)$ 
(\mbox{example \ref{cechfamily}).} A locally constant object is an
object $X$ together with a trivialization structure $\theta$. This
structure consists in a family of isomorphisms $\{\theta_i\}_{i \in
  I\,}$: 
$$
\xymatrix
         {
           \gamma^{*}R_{i} \times U_{i} \ar[rr]^{\theta_{i}}
                                        \ar[rd] 
       & &  X \times U_{i}  \ar[ld]
       \\
         &  U_i
         }
$$
 where 
$R \mr{} I$, $\{R_i\}_{i \in I}$ is a family of sets. These objects
 are constructed by  descent (see \cite{G1}) on a $U_\bullet$-descent
 datum $\sigma$ on an object in $\cc{E}_{/U}$ of the form  
\mbox{$ \gamma^{*}R \, \times_{\gamma^{*}I} \, U  \mr{} U$}, 
$\{\gamma^{*}R_i \times U_i \mr{} U_i\}_{i \in I}$ (definition
\ref{descentf}). Such a descent datum    
consists of a
family of isomorphisms,  
$\{\sigma_{j,\,i}\}_{(i, \, j) \in N_1}$:
$$
\xymatrix
         {
            \gamma^{*}R_i \times U_i \times U_j  \ar[r]^{\sigma_{j,\,i}}
                                                 \ar[d]
          & \gamma^{*}R_j \times U_j \times U_i  \ar[d] 
          \\
            U_i \times U_j  \ar[r]^{\tau}  
          & U_j \times U_i
         }
$$
satisfying the corresponding  identity and cocycle conditions in
\mbox{definition \ref{descentf}.}
% 
%The commutativity of the square implies that $\sigma_{j,\,i}$ is completely determined by its first projection.
% 
%We abuse notation and write this projection with the same
%letter: $\gamma^{*}S_i \times U_i \times U_j  \mr{\sigma_{j,\,i}} \gamma^{*}S_j$.
%
The relationship between the trivialization structure $\theta$ and the descent datum $\sigma$ is given in the following commutative diagram:
$$
\xymatrix
         {
            \gamma^{*}R_i \times U_i \times U_j 
                               \ar[r]^{\sigma_{j,\,i}}
                               \ar[d]^{\theta_i \times U_j}
          & \gamma^{*}R_j \times U_j \times U_i  
                               \ar[d]^{\theta_j \times U_i} 
          \\
            X \times U_i \times U_j  \ar[r]^{X \times \tau}  
          & X \times U_j \times U_i
         }
$$

Let $X  \sim  (R \to I, \, \sigma)$ be a locally constant object
determined by a \mbox{$U_\bullet$-descent} datum $\sigma$ on a cover $U \to \gamma^*
I$. Recall from \cite{D3} the following  
\begin{definition} \label{actionspan}
$
\xymatrix@R= 1.8ex
         {
          \\
 \textrm{An \emph{action span} for $X$ is a span $\; \ell \;= \;$}
         }        
\xymatrix@C=3ex@R=1ex
         {
          {} & V \ar[dl]_{u} \ar[dr]^{v}
         \\
         U_i & {} & U_j\,,\
         }
$
with  $V \neq \emptyset$, and such that there is a bijection $S_i
\mr{s_\ell} S_j$ (necessarily unique) such that the following diagram
commutes  
$$
\xymatrix
         {
            \gamma^{*}R_i \times U_i \times U_j 
                               \ar[r]^{\sigma_{j,\,i}}
          & \gamma^{*}R_j \times U_j \times U_i   
          \\
             \gamma^{*}R_i \times V  
                               \ar[u]_{\gamma^{*}S_i \times (u,\, v)}
                               \ar[r]^{\gamma^* s_\ell \times V}  
          &  \gamma^{*}R_j \times V
                               \ar[u]_{\gamma^{*}S_j \times (v,\, u)}
         }
$$
\end{definition}
\begin{remark} \label{sieve}
Given a morphism of non empty spans $\ell' \to \ell$, $V' \to V$ 
%as in (2) proposition \ref{groupoid}
, if $\ell$ is an action span, then so
it is $\ell'$, and $s_{\ell'} = s_\ell$. 
\qed\end{remark}

\begin{definition} \cite[2.12]{D3} \label{coveringprojectionobject}
We say that a locally constant object \mbox{$X \sim (S \to I, \, \sigma)$}
trivialized by a cover $U \to \gamma^* I$ is a \emph{covering
  projection} if, for each \mbox{$(i, \, j) \in N_1$}, the family
$V \mr{(u, \, v)} U_i \times U_j$ is an epimorphic family, where 
$(V,\, u, \, v)$ ranges over all action spans. By remark \ref{sieve}
it is equivalent to restrict $V$ to a site of definition. 
\end{definition}
Every span with connected vertex is an action span, thus we have:

\vspace{-1ex}

\begin{proposition} \label{cparesdatum}
In a locally connected topos every locally constant object is a
covering projection. \cqd
\end{proposition}
\begin{proposition} \label{facts}
Let $X \sim (S \to I, \, \sigma)$ be a locally constant object
trivialized by a cover $U \to \gamma^* I$, let $\cc{C}_{sp}$ be the set
of all action spans with $V$ in a site of 
definition, and $\cc{C}$ be the set of vertices of the spans in
$\cc{C}_{sp}$. Then: 
  
(1) The conditions in construction \ref{1spanref} are satisfied.

(2) The map $H_2 \mr{} (cosk_1 H_\bullet)_2$ is an epimorphism.

(3) If $X$ is a covering projection, then the map 
$H_1 \mr{} (cosk_0 H_\bullet)_1$ is an epimorphism. 
Thus $H_\bullet \mr{h_\bullet} U_\bullet$ is an hypercovering. 
\end{proposition}
\begin{proof}
(1) Observe that the dual span of an action span is an action span
  with inverse bijection $s_\ell^{-1}$. The other requirements follow
  from the descent identity condition (1) in definition \ref{descentf}
  and remark \ref{sieve}. Recall
  that in this case $(s_0)_i$ is the diagonal of $U_i$.

We refer now to proposition \ref {hypercovering}, \ref{1indexedhypercover}:  

(2) From remark \ref{sieve} 
  it follows that for any map $\emptyset \neq 
V \mr{} P_{\ell t r}$, $V$ is the vertex of a (in fact three) action spans, 
thus it is in $\cc{C}$.
 
(3) It holds by definition of covering projection.
\end{proof}
\begin{remark}
From remark \ref{sieve} it follows that the simplicial
family \mbox{$H_\bullet \mr{} \gamma^*(S_\bullet)$} is a \emph{sieve} in the
sense that given any $V \in \cc{C}$, $V \subset (H_1)_\ell$, there
exists $t \in S_1$ such that $(H_1)_t = V$.  
\end{remark}
Finally, from remark \ref{sieve} and propositions \ref{constructionG},
\ref{hyper=all} and \ref{cparesdatum} we have:
\begin{theorem} \label{main1}
Let $X \sim (R \to I, \, \sigma)$ be a covering projection trivialized
by a cover $U \to \gamma^* I$. Then there exist a self-dual simplicial
hypercover 
refinement of the Cech simplicial family (see proposition \ref{counit})
$$
\xymatrix
        {
          H_\bullet \ar[r]^{h_\bullet} \ar[d]^{\zeta_\bullet} 
        & U_\bullet \ar[d]^{\zeta_\bullet}
        \\
          \gamma^{*} S_\bullet \ar[r]^{\gamma^*\alpha_\bullet} 
        & \gamma^{*} N_\bullet
        }
$$ 
satisfying condition G (definition \ref{conditionG}), and a consistent
(cf \ref{consistent}) 
\mbox{$S_\bullet$-descent} datum $\{s_\ell \}_{\ell \in S_1}$ on the 
family $R \to I$ such that the corresponding \mbox{$H_\bullet$-descent
datum} $\sigma_\ell$ (proposition \ref{hyper=all}) is of the form
\mbox{$\sigma_\ell = \gamma^*(s_\ell)\times (\tau_1)_\ell$} (remark
\ref{hyper=allremark}). Vice-versa, any such descent datum on a
self-dual simplicial hypercover 
refinement of the Cech simplicial family determines a covering
projection trivialized by the cover $U \to \gamma^* I$.
\cqd\end{theorem}

\emph{We will also say that the covering projection is trivialized by the
hypercover $H_\bullet \mr{} \gamma^*(S_\bullet)$}
With the notation in the theorem above, from proposition
\ref{descent=actionG} we have: 

\begin{theorem} \label{main2}
Given a cover $U \to \gamma^* I$ and a self-dual simplicial hypercover
refinement of the Cech simplicial family 
satisfying condition G, then the category of covering projections
trivialized by a consistent 
\mbox{$S_\bullet$-descent} datum $\{s_\ell \}_{\ell \in S_1}$ on a
family $R \to I$, is isomorphic to the category (topos) of  left actions of the
G-fundamental groupoid of the family (Proposition \ref{groupoid}).
\cqd\end{theorem}

In the case of a locally connected topos, taking into account
construction \ref{0spanref}, example 
\ref{connectedhypercover} and proposition \ref{cparesdatum} it
follows: 

\begin{theorem}
Given any locally connected topos $\cc{E}$, the statement in theorem
\ref{main1} holds for any locally constant object $X \sim (R \to I, \,
\sigma)$ trivialized 
by a cover $U \to \gamma^* I$. Thus X can be constructed by a  
\mbox{$S_\bullet$-descent} datum $\{s_\ell \}_{\ell \in S_1}$ on the 
family $R \to I$, where $S_\bullet$ is the simplicial set constructed
in \ref{0spanref} with any set of connected generators.   
\cqd\end{theorem}

\section{Fundamental progroupoid of a topos} 
This section is rather sketchy and we refer the reader to \cite{D3} for details and complete proofs.  
%Given a cover family 
%$\cc{U} = U \mr{} \gamma^*S$, there is a faithful (but not full) functor
% $\cc{G}_\cc{U} \mr{} \cc{E}$ from the category of covering projections constant on $\cc{U}$. With this,  
%in \cite{D3} we have shown that the category $\cc{G}_\cc{U}$ of covering projections trivialized by $\cc{U}$ is an atomic topos, and then (by an application of Theorem 1 in \cite{JT} VIII, 3.) it follows it is the classifying topos of a localic groupoid $\nn{G}_\cc{U}$, $\cc{G}_\cc{U} = \beta\nn{G}_\cc{U}$. There is a faithful (but not full) functor
 %$\cc{G}_\cc{U} \mr{} \cc{E}$, 
%it is easy to construct
% the colimit $c\cc{G}(\cc{E})$ (inside $\cc{E}$) of the categories $\cc{G}_\cc{U}$ indexed by the covers families,    
%$\cc{G}_\cc{U} \to c\cc{G}(\cc{E}) \to \cc{E}$. The category
% $c\cc{G}(\cc{E})$ is the category whose objects are all covering
% projections. Taking small coproducts of covering projections
% determines a topos  $\cc{G}(\cc{E})$ together with a faithful functor 
%$\cc{G}(\cc{E}) \mr{} \cc{E}$ inverse image of a geometric morphism.
%
%\vspace{1ex}
%
Given a cover family $\cc{U} = U \mr{} \gamma^*S$
and a  family hypercover refinement
\mbox{$\cc{H} = H_\bullet \mr{} \gamma^*(S_\bullet)$,} 
%as in theorem \ref{main1}, 
there is clear definition of morphisms of covering projections constant on $\cc{U}$ (constant on $\cc{H}$), \cite[1.4]{D3}. This determines categories 
$\cc{G}_\cc{U}$ ($\cc{G}_\cc{H}$), furnished with a faithful (but not full) functor
 $\cc{G}_\cc{U} \mr{} \cc{E}$ ($\cc{G}_\cc{H} \mr{} \cc{E}$). With this it is easy to construct
the colimit (inside $\cc{E}$) of the categories $\cc{G}_\cc{U}$  
($\cc{G}_\cc{H}$) indexed by $\cc{U}$ (indexed by $\cc{H}$), \cite[1.4]{D3}. It follows from \mbox{theorem \ref{main1}} that every covering projection is in some $\cc{G}_\cc{H}$, so these two colimits are equal. We  denote this category $c\cc{G}(\cc{E})$, it is the category of all covering projections. We have  
$\cc{G}_\cc{H} \to c\cc{G}(\cc{E}) \to \cc{E}$.
It follows from theorem 
\ref{main2} that the category
$\cc{G}_\cc{H}$  is the
classifying topos of a groupoid $\nn{G}_\cc{H}$ (the
G-fundamental groupoid of the simplicial family)  $\cc{G}_\cc{H} =
 \beta\nn{G}_\cc{H}$. This determines a protopos 
 $\cc{G}(\cc{E}) = \{\cc{G}_\cc{H}\}_\cc{H}$ and a progroupoid
$\pi_1(\cc{E}) = \{\nn{G}_\cc{H}\}_\cc{H}$, and we have 
$\beta \pi_1(\cc{E}) = \cc{G}(\cc{E})$. The inverse limit topos of this protopos is the topos of sheaves for a subcanonical Grothendieck topology on the category $c\cc{G}(\cc{E})$ \cite[4.1, 4.4]{D3}. 

 Given a group $K$, recall that a
\emph{K-torsor} in a 
topos $\cc{E}$ is an \mbox{object} $T \in \cc{E}$, $T \to 1$ epi, together with
an action $\gamma^*K 
\times T \mr{} T$ such that the arrow $\gamma^*K
\times T \mr{\varepsilon} T \times T$ defined by
 $\varepsilon(x, \, u) = (x \pa u, \, u)$ is an isomorphism.
 Clearly any
torsor $T$ determines in a canonical way a locally constant object
 $T = (T,\,K,\, \varepsilon)$ split by the (singleton family) cover $T \to 1$, which in fact is a covering projection.
  Following exactly the same lines that in \cite[Section 6]{D3}, it can be proved that $\pi_1(\cc{E})$ 
%is the fundamental progroupoid of the topos $\cc{E}$ in the sense that it 
classifies torsors. 
If we
denote $pro\cc{G}rpd$, the 2-category of progroupoids, we have: 

\vspace{1ex}

\emph{There is an equivalence of categories
 $pro\cc{G}rpd[\pi_1(\cc{E}),\, K] 
\cong K \textrm{-}\cc{T}ors(\cc{E})$.}

\vspace{1ex}
 
Note that this
furnish an explicit construction of the fundamental progroupoid
$\pi(\cc{E})$, to be compared in the locally connected case with the construction in \cite[Section 10]{AM}.

\end{document}